\documentclass[11pt, a4paper]{amsart}

\usepackage{amsfonts, amsthm, amsmath, amscd, amssymb}
\usepackage{hyperref}
\usepackage{tikz-cd}

\newcommand\Kbar{\overline{K}}
\newcommand\frakt{\mathfrak{t}}
\newcommand\frakg{\mathfrak{g}}
\newcommand\A{\mathcal{A}}
\newcommand\OO{\mathcal{O}}
\newcommand\CC{\mathbb{C}}

\newcommand\TT{\mathcal{T}}
\newcommand\GG{\mathcal{G}}

\newcommand\ZZ{\mathbb{Z}}
\newcommand\PP{\mathbb{P}}

\DeclareMathOperator{\Pic}{Pic}
\DeclareMathOperator{\GL}{GL}
\DeclareMathOperator{\Aut}{Aut}

\DeclareMathOperator{\Hom}{Hom}
\DeclareMathOperator{\Gal}{Gal}
\DeclareMathOperator{\Ad}{Ad}
\DeclareMathOperator{\Res}{Res}
\DeclareMathOperator{\Lie}{Lie}
\newcommand\Gm{\mathbb{G}_\mathrm{m}}
\newcommand\GmK{\mathbb{G}_{\mathrm{m},K}}
\newcommand\GmL{\mathbb{G}_{\mathrm{m},L}}
\newcommand\id{\mathrm{id}}
\newcommand\op{\mathrm{op}}
\newcommand\conj{\mathrm{conj}}

\newtheorem{theorem}{Theorem}[section]
\newtheorem*{thm}{Theorem}
\newtheorem{lemma}[theorem]{Lemma}
\newtheorem{proposition}[theorem]{Proposition}
\newtheorem{cor}[theorem]{Corollary}
\theoremstyle{definition}
\newtheorem{definition}[theorem]{Definition}
\newtheorem{remark}[theorem]{Remark}

\begin{document}

\title
[$G$-torsors and universal torsors over del Pezzo surfaces]
{$G$-torsors and universal torsors\\ over nonsplit del Pezzo surfaces}

\author{Ulrich Derenthal}

\address{Institut f\"ur Algebra, Zahlentheorie und Diskrete
  Mathematik, Leibniz Universit\"at Hannover, Welfengarten 1, 30167
  Hannover, Germany}

\address{School of Mathematics, Institute for Advanced Study, 1 Einstein Drive, Princeton, New Jersey, 08540, USA}
 
\email{derenthal@math.uni-hannover.de}

\author{Norbert Hoffmann}

\address{Department of Mathematics and Computer Studies, Mary
  Immaculate College, South Circular Road, Limerick, Ireland}

\email{norbert.hoffmann@mic.ul.ie}

\date{\today}



\begin{abstract}
  Let $S$ be a smooth del Pezzo surface that is defined over a field $K$ and
  splits over a Galois extension $L$. Let $G$ be either the split reductive
  group given by the root system of $S_L$ in $\Pic S_L$, or a form of it
  containing the N\'eron--Severi torus. Let $\GG$ be the $G$-torsor over $S_L$
  obtained by extension of structure group from a universal torsor $\TT$ over
  $S_L$.  We prove that $\GG$ does not descend to $S$ unless $\TT$ does. This
  is in contrast to a result of Friedman and Morgan that such $\GG$ always
  descend to singular del Pezzo surfaces over $\CC$ from their
  desingularizations.
\end{abstract}

\maketitle

\section{Introduction}\label{sec:intro}

One of the most famous theorems of 19th century algebraic geometry
states that cubic surfaces over $\CC$ contain precisely $27$
lines. The symmetry group of their configuration is the Weyl group $W$
of the root system $\mathbf{E}_6$.

Cubic surfaces over fields $K$ that are not algebraically closed exhibit a
substantially wider range of phenomena.  Given such a cubic surface $S$ over
$K$, there is a finite Galois extension $L/K$ over which $S$ splits (i.e.,
over which all $27$ lines are defined). Its Galois group $\Gamma$ acts on the
configuration of lines via $W$. This action plays a major role in the geometry
of $S$.

An important tool in the study of such cubic surfaces $S$ are the
universal torsors introduced and studied by Colliot-Th\'el\`ene and
Sansuc \cite{MR89f:11082}. These are certain torsors over $S$ under
the N\'eron--Severi torus.  Such universal torsors do not always exist
in the nonsplit case, and they do not descend to singular cubic
surfaces from their desingularizations. In particular, this does not
provide degenerations of universal torsors to singular cubic surfaces.

Friedman and Morgan suggested to replace the N\'eron--Severi torus here by a
reductive algebraic group $G$ of type $\mathbf{E}_6$. They construct natural
$G$-torsors over singular cubic surfaces over $\CC$ \cite{MR1941576}.  These
$G$-torsors are compatible with degenerations of split cubic surfaces
\cite{MR4188698}.

In this paper, we deal with the existence of natural $G$-torsors over nonsplit
smooth cubic surfaces.  Here, the natural candidates for $G$ are a certain split
group of type $\mathbf{E}_6$, and forms of it that contain the N\'eron--Severi
torus, because both reflect the combinatorial information coming from
$S$. Motivated by the result of Friedman and Morgan, there has been a hope
that such $G$-torsors could exist also in situations where universal torsors
do not exist.  However, the main result of the present paper is that such
$G$-torsors do not exist unless universal torsors exist, and even if they
exist, they are not more canonical than universal torsors.

\medskip

More generally, let $S$ be a smooth del Pezzo surface over a field
$K$.  By \cite[Theorem~1.6]{MR3114932}, there is a finite Galois
extension $L/K$ such that the base change $S_L$ is split.  Then the
Galois group $\Gamma=\Gal(L/K)$ acts on $\Lambda_S:=\Pic S_L$.

Manin \cite[\S 25]{MR833513} discovered that the Picard group $\Lambda_S$ of the
split del Pezzo surface $S_L$ together with the intersection form
$(\_\,,\_)$ provide us with a reduced root datum
$(\Phi_S \subset \Lambda_S,\Phi_S^\vee \subset \Lambda_S^\vee)$ given as follows: The set of
roots is
\begin{equation}\label{eq:Phi}
  \Phi_S := \{ \alpha \in \Lambda_S \mid (\alpha,\alpha)=-2,\ (\alpha,-K_{S_L})=0\}.
\end{equation}
With $\alpha^\vee \in \Lambda_S^\vee$ given by
$\langle\alpha^\vee,\lambda\rangle := (-\alpha,\lambda)$ for
$\lambda \in \Lambda_S$, the set of coroots is
\begin{equation*}
  \Phi_S^\vee:= \{\alpha^\vee \in \Lambda_S^\vee \mid \alpha \in \Phi_S\}.
\end{equation*}
Its Weyl group $W_S \subset \Aut\Lambda_S$ is the symmetry group of the configuration of
$(-1)$-curves on $S_L$.

By a $\Gamma$-twisted root datum, we mean a root datum
$(\Phi \subset \Lambda, \Phi^\vee \subset \Lambda^\vee)$ together with
a left action $\Gamma \times \Lambda \to \Lambda$ such that
$\gamma(\Phi) \subset \Phi$ and
$\gamma^*(\Phi^\vee) \subset \Phi^\vee$ for all $\gamma \in \Gamma$;
for example, each reductive algebraic group over $K$ that splits over
$L$ yields a $\Gamma$-twisted root datum (cf.\ \cite[Expos\'e XXII,
D\'efinition~1.9 and Proposition~1.10]{SGA3}, and
\cite{MR3260846} for a discussion of the converse direction). Let
\begin{equation*}
  (\Phi_S \subset \Lambda_S,\Phi_S^\vee \subset \Lambda_S^\vee,
  \Gamma \times \Lambda_S \to \Lambda_S)
\end{equation*}
be the $\Gamma$-twisted root datum given by
$(\Phi_S \subset \Lambda_S,\Phi_S^\vee \subset \Lambda_S^\vee)$
endowed with the action $\Gamma \times \Lambda_S \to \Lambda_S$ coming
from the fact that $S$ is defined over $K$.

The N\'eron--Severi torus for $S_L$ is the split $L$-torus $T$ with character
lattice $\Hom(T,\GmL) = \Lambda_S$ as abelian groups.  A universal torsor for
$S_L$ is a $T$-torsor $\TT$ over $S_L$ whose type $\Hom(T,\GmL) \to \Pic S_L$
(in the sense of \cite{MR89f:11082}, see Section~\ref{sec:split}) is the
identity on $\Lambda_S$. Such a $\TT$ exists and is unique up to isomorphism.

The N\'eron--Severi torus for $S$ is the $K$-torus $T$ that splits
over $L$ with character lattice $\Hom(T_L,\GmL) = \Lambda_S$ as
$\Gamma$-modules. A universal torsor for $S$ is a $T$-torsor $\TT$
over $S$ whose base change $\TT_L$ is a universal torsor for
$S_L$. Such a $\TT$ exists if $S(K) \ne \emptyset$
\cite[Remarque~2.2.9]{MR89f:11082}, but otherwise may or may not exist
\cite[Exemples~2.2.11 and 2.2.12]{MR89f:11082}.  For example, the
blow-up of a Severi--Brauer surface $B$ of index $3$ in an effective
$0$-cycle of degree $6$ is a cubic surface $S$ without a universal
torsor, since the relevant elementary obstructions
\cite[D\'efinition~2.2.1 and Proposition~2.2.8(iii)]{MR89f:11082} for
$B$ and for $S$ coincide.  If universal torsors exist, they may or may
not be unique up to isomorphism \cite[(2.0.2)]{MR89f:11082}.

In this situation, our main result on the existence and classification of such
$G$-torsors on del Pezzo surfaces $S$ is:

\begin{thm}
  Let $G$ be a reductive group over $K$ with maximal torus
  $\iota: T \hookrightarrow G$ such that
  \begin{itemize}
  \item[(i)] $T$ is split over $K$ with character lattice
    $\Hom(T,\GmK)=\Lambda_S$, and the root datum given by
    $G \supset T$ is
    $(\Phi_S \subset \Lambda_S,\Phi_S^\vee \subset \Lambda_S^\vee)$,
    or
  \item[(ii)] $T$ is the N\'eron--Severi torus for $S$, and the
    $\Gamma$-twisted root datum given by $G \supset T$ is 
    $(\Phi_S \subset \Lambda_S,\Phi_S^\vee \subset \Lambda_S^\vee,
    \Gamma \times \Lambda_S \to \Lambda_S)$.
  \end{itemize}
  
  Let $\TT$ be a universal $T_L$-torsor over $S_L$, and let
  $\GG:=(\iota_L)_*\TT$ be the $G_L$-torsor over $S_L$ obtained by extension
  of structure group.
  Then the groupoid of $G$-torsors $\GG^\circ$ over $S$ such that
  $\GG^\circ_L \cong \GG$ is equivalent to the
  groupoid of universal torsors over $S$.
\end{thm}

In particular, $\GG$ descends to a $G$-torsor over $S$ if and only if
universal torsors for $S$ exist; if this is the case, then the $K$-forms of
$\GG$ are in bijection to universal torsors.  Part (ii) of this theorem is
contained in Theorem~\ref{thm:main_category} below, and part (i) is contained
in its Corollary~\ref{cor:main_split}.

The proof of Theorem~\ref{thm:main_category} is based on descent from the
split situation. It turns out that the relevant $T$-torsors and $G$-torsors
have the same descent data because they have the same group of global
automorphisms, essentially by Proposition~\ref{prop:IsomT_IsomG},
which might be of independent interest since it applies to torsors under
general reductive groups $G$ over smooth projective schemes.
Corollary~\ref{cor:main_split} then follows using that the split group $G$
over $K$ contains a N\'eron--Severi torus, a result due independently to Gille
\cite{MR2139505} and Raghunathan \cite{MR2125504}.

For another connection between the del Pezzo surface $S$ and this reductive
group $G$ via universal torsors, see \cite{MR2368955,MR2976314} over
algebraically closed fields, and \cite[Theorem~4.4]{MR2576802} over nonclosed fields.
Regarding the Hasse principle for the existence of rational points and weak
approximation, torsors under connected linear algebraic groups $G$ do not
provide finer obstructions than the Brauer--Manin obstruction
\cite[Th\'eor\`eme~2]{MR1905103}.

The structure of our paper is as follows. In Section~\ref{sec:torsors}, we
collect a few basic facts about $G$-torsors, with a view towards their
descent. We circumvent the difficulty of two Galois actions (on $S_L$ and
$G_L$) as follows: Instead of working with $G_L$-torsors over $S_L$, we view
them as $G$-torsors over $S_L$ regarded as a $K$-scheme.  In
Section~\ref{sec:split}, we compare automorphisms of $T$-torsors and
$G$-torsors in the split case. Finally, in Section~\ref{sec:proof}, we apply
our descent setup to prove the main results.

\subsection*{Acknowledgments}

We learned about this question from Yuri Tschinkel. We are grateful to
the anonymous referee for several useful remarks and suggestions that
improved the paper. The first author was partly supported by grant DE
1646/4-2 of the Deutsche Forschungsgemeinschaft. Some of this work was
done while the second author was on sabbatical leave at the Riemann
Center for Geometry and Physics of Leibniz Universit\"at Hannover,
supported by Mary Immaculate College Limerick through the PLOA
program. During revisions, the first author was supported by
the Charles Simonyi Endowment at the Institute for Advanced Study.

\section{Torsors}\label{sec:torsors}

Let $X$ be a smooth projective scheme over $K$. We assume that $X$ is
connected (but not necessarily geometrically connected over $K$). Let $G$ be a linear
algebraic group over $K$.

\begin{definition}\label{def:G-torsor}
  A \emph{$G$-torsor over $X$} is a faithfully flat morphism $\GG \to X$ of
  finite type together with a right action $\GG \times_K G \to \GG$ over $X$
  such that the map $\GG \times_K G \to \GG \times_X \GG$ given by $(u,g)
  \mapsto (u,ug)$ is an isomorphism.
\end{definition}

\begin{definition}\label{def:ext_of_structure_group}
  Let $\phi: G \to H$ be a homomorphism of linear algebraic groups over
  $K$. The \emph{extension of structure group of $\GG$ along $\phi$} is
  the $H$-torsor
  \begin{equation*}
    \phi_*\GG := \GG \times^G H = (\GG \times_K H) / G \to X,
  \end{equation*}
  where $G$ acts on $\GG$ from the right, and on $H$ via $\phi$ from the left.
\end{definition}

Let $f : Y \to X$ be a morphism of connected smooth projective $K$-schemes. 

\begin{definition}\label{def:pullback}
  The \emph{pullback of $\GG$ along $f$} is the $G$-torsor
  \begin{equation*}
    f^*\GG := \GG \times_X Y \to Y.
  \end{equation*}
\end{definition}

\begin{remark}\label{rem:ext}
  Generalizing Definition~\ref{def:ext_of_structure_group}, if $\phi$ is a
  homomorphism of group schemes over $X$ from $G_X:= G\times_K X$ to
  $H_X:=H \times_K X$, then
  \begin{equation*}
    \phi_*\GG := (\GG \times_X H_X)/G_X
  \end{equation*}
  is still an $H$-torsor over $X$. Its pullback along $f$ satisfies
  \begin{equation}\label{eq:pullback_extension}
    f^*(\phi_*\GG) \cong (f^*\phi)_*(f^*\GG),
  \end{equation}
  where $f^*\phi: G_Y \to H_Y$ is the pullback of $\phi$.
\end{remark}

\begin{definition}\label{def:galois_mor}
  A morphism $f : Y \to X$ between connected smooth projective schemes over
  $K$ is a \emph{Galois covering} if $f$ is finite, surjective and \'etale such that the
  natural map
  \begin{equation*}
    \Sigma \times Y \to Y \times_X Y, \quad (\sigma,y) \mapsto (\sigma y, y)
  \end{equation*}
  is an isomorphism for the discrete group $\Sigma := \Aut_X Y$ \cite[\S 6.2B]{MR1045822}.
\end{definition}

Here, each $\sigma \in \Sigma$ is a $K$-morphism $\sigma : Y \to Y$. Hence any
$G$-torsor $\GG$ over $Y$ can be pulled back along $\sigma$, resulting in
another $G$-torsor $\sigma^*\GG$ over $Y$.

Suppose that $\GG = f^*\GG^\circ$ for some $G$-torsor $\GG^\circ$ over
$X$. For each $\sigma \in \Sigma$, let $c_\sigma : \sigma^*\GG \to \GG$ be the
canonical isomorphism
$\sigma^*\GG = \sigma^*f^*\GG^\circ \to f^*\GG^\circ = \GG$ of $G$-torsors
over $Y$ coming from the fact that $f \circ \sigma = f : Y \to X$. By
construction, they satisfy the cocycle condition
\begin{equation}\label{eq:cocycle}
  c_{\varrho\sigma}=c_\sigma \circ \sigma^*(c_\varrho)
\end{equation}
as isomorphisms
$(\varrho\sigma)^*\GG = \sigma^*\varrho^*\GG \to \sigma^*\GG \to \GG$ of
$G$-torsors over $Y$ for every $\varrho,\sigma \in \Sigma$.  Conversely, we
have the following descent result.

\begin{lemma}\label{lem:descent_torsors}
  Given a $G$-torsor $\GG$ over $Y$, and isomorphisms
  $c_\sigma : \sigma^*\GG \to \GG$ for $\sigma \in \Sigma$ satisfying
  \eqref{eq:cocycle} for all $\varrho,\sigma \in \Sigma$, there is a unique
  $G$-torsor $\GG^\circ$ over $X$ such that $f^*\GG^\circ$ with the induced
  cocycle is isomorphic to $\GG$ with $(c_\sigma)_{\sigma \in \Sigma}$.
\end{lemma}

\begin{proof}
  The relatively affine scheme $\GG$ over $Y$ is the relative spectrum of a
  quasi-coherent $\OO_Y$-algebra $\A$, which descends to a unique quasi-coherent
  $\OO_X$-algebra $\A^\circ$ according to
  \cite[\href{https://stacks.math.columbia.edu/tag/0D1V}{Lemma~0D1V}]{stacks-project}.
  We define $\GG^\circ$ as the relative spectrum of $\A^\circ$ over $X$.

  The group action $\GG \times_K G \to \GG$ corresponds to a morphism
  $\A \to \A \otimes_K K[G]$ of quasi-coherent $\OO_Y$-algebras, where
  $K[G]:=\Gamma(G,\OO_G)$ is the coordinate ring of $G$. Since $\sigma$ is a
  $K$-morphism, we have a canonical isomorphism
  \begin{equation*}
    \sigma^*(\A \otimes_K K[G]) = (\sigma^*\A) \otimes_K K[G],
  \end{equation*}
  which allows us to use
  \cite[\href{https://stacks.math.columbia.edu/tag/0D1V}{Lemma~0D1V}]{stacks-project}
  again to descend the given group action to a unique group action
  $\GG^\circ \times_K G \to \GG^\circ$. This turns $\GG^\circ$ into a $G$-torsor over $X$, as
  required.
\end{proof}

\section{Automorphisms of $T$-torsors and $G$-torsors}\label{sec:split}

Let $X$ be a smooth projective scheme over a field $K$.  Let
$L_1,\dots,L_n$ be line bundles over $X$ such that $\Hom(L_i,L_j)=0$
for $i \ne j$, and hence in particular $L_i \not\cong L_j$ for
$i \ne j$. Then every automorphism of the vector bundle
$E:=L_1 \oplus \dots \oplus L_n$ over $X$ is of the form
$c_1 \oplus \dots \oplus c_n$ for automorphisms $c_i$ of $L_i$. This
is the special case $G = \GL_n$ of the following proposition.

\begin{proposition}\label{prop:IsomT_IsomG}
  Let $G$ be a reductive group over $K$ with split maximal torus
  $\iota: T \hookrightarrow G$ and resulting Weyl group $W_G$. Let $\TT$ be a
  $T$-torsor over $X$ satisfying the following two conditions:
  \begin{enumerate}
  \item[(i)] For every root $\alpha : T \to \Gm$ of $G$, the line bundle
    $L_\alpha$ given by the $\Gm$-torsor $\alpha_*\TT$ over $X$ has no nonzero
    global section.
  \item[(ii)] For every $1 \ne w \in W_G$, the $T$-torsor $w_*\TT$ obtained by
    extension of structure group along $w : T \to T$ is not
    isomorphic to $\TT$.
  \end{enumerate}
  Let $\GG := \iota_* \TT$ denote the $G$-torsor over $X$ obtained from $\TT$
  by extension of structure group. Viewing the total space $\TT$ as contained
  in $\GG$, every automorphism of the $G$-torsor $\GG \to X$ restricts to an
  automorphism of the $T$-torsor $\TT \to X$.
\end{proposition}

\begin{proof}
  Since $T$ is split over $K$, the assumptions (i) and (ii) still hold after
  base change to the algebraic closure $\Kbar$ of $K$. This allows us to
  assume $K=\Kbar$ without loss of generality. Working on each connected
  component of $X$ individually, we may also assume that $X$ is connected.

  We consider the natural map
  \begin{equation*}
    T=\Aut\TT \to \Aut\GG.
  \end{equation*}
  Here the group scheme of global automorphisms
  $\Aut\GG$ over $K$ is by definition the Weil restriction
  $\Res_{X/K}(\Aut_X\GG)$ of the group scheme of local automorphisms
  \begin{equation*}
    \Aut_X\GG := \GG \times^{G,\,\Ad} G
  \end{equation*}
  over $X$, where $G$ acts via the homomorphism
  \begin{equation*}
    \Ad : G \to \Aut G,\quad g \mapsto \conj_g := g\cdot \_ \cdot g^{-1},
  \end{equation*}
  and $\Aut\TT$ is defined similarly.  These Weil restrictions exist
  as linear algebraic group schemes over $K$ according to \cite[\S
  1.4]{MR0263826}.

  Let $K[\epsilon]$ be the ring of dual numbers, with $\epsilon^2=0$.
  Since our map
  \begin{equation*}
    T=\Aut\TT \to \Aut\GG
  \end{equation*}
  is injective on $K$-points and on $K[\epsilon]$-points, it is a
  closed immersion by \cite[Expos\'e VI\textsubscript{B},
  Corollaire~1.4.2]{SGA3}. We have to prove that it is an isomorphism.

  Let $\frakg$ denote the Lie algebra of $G$. Let $G_{K[\epsilon]}$ be
  the base change of $G$ from $K$ to $K[\epsilon]$.  Its Weil
  restriction
  \begin{equation*}
    G[\epsilon]:=\Res_{K[\epsilon]/K}(G_{K[\epsilon]})
  \end{equation*}
  is a smooth linear algebraic group and appears in the natural short
  exact sequence
  \begin{equation*}
    0 \to \frakg \to G[\epsilon] \to G \to 1
  \end{equation*}
  by \cite[II, \S4, Th\'eor\`eme~3.5]{MR0302656}.
  This induces an exact sequence 
  \begin{equation*}
    0 \to \GG \times^{G} \frakg \to \GG \times^{G} G[\epsilon]
    \to \Aut_X\GG \to 1
  \end{equation*}
  of associated group schemes over $X$. Since the Lie algebra
  $\Lie(\Aut\GG)$ of $\Aut\GG$ consists of its $K[\epsilon]$-valued
  points (i.e., sections from $X$ into $\GG \times^{G} G[\epsilon]$)
  that reduce to the identity modulo $\epsilon$, we deduce
  \begin{equation*}
    \Lie(\Aut\GG) = H^0( X, \GG \times^{G} \frakg).
  \end{equation*}
  Since $\GG = \iota_* \TT$, we thus obtain
  \begin{equation*}
    \Lie(\Aut\GG) = H^0( X, \TT \times^T \frakg).
  \end{equation*}

  Let
  \begin{equation*}
    \frakg = \frakg_0 \oplus \bigoplus_{\alpha} \frakg_{\alpha}
  \end{equation*}
  be the root space decomposition of $\frakg$ into its eigenspaces
  under $T$.  For any root $\alpha$, the line bundle
  \begin{equation*}
    \TT \times^T \frakg_{\alpha} \cong L_{\alpha}
  \end{equation*}
  has no nonzero global sections by assumption (i).  Therefore,
  \begin{equation*}
    H^0( X, \TT \times^T \frakg) = H^0( X, \TT \times^T \frakg_0)
      = H^0( X, \OO_{X} \otimes \frakg_0) = \frakg_0 = \frakt,
  \end{equation*}
  the Lie algebra of $T$. Since
  $\dim \Aut\GG \ge \dim T = \dim\frakt$, it follows that $\Aut\GG$ is
  smooth with tangent spaces $\frakt$.  This shows that our closed
  immersion $T \hookrightarrow \Aut\GG$ is an open embedding.

  Being also connected, $T$ is the connected component of the identity
  in $\Aut\GG$. In particular, $T$ is normal in $\Aut\GG$.  We view
  elements of $\Aut \GG$ as global sections of 
  \begin{equation*}
    \Aut_X\GG = \GG \times^{G,\,\Ad} G = \TT \times^{T,\,\Ad} G
    \longrightarrow X.
  \end{equation*}
  Then the elements of $T \subset \Aut\GG$ become the constant global sections
  of
  \begin{equation*}
    X \times T = \TT \times^{T,\,\Ad} T \subset \TT \times^{T,\,\Ad} G.
  \end{equation*}

  Let $\TT_x$ be the fiber of $\TT$ over a point $x \in X$. Since each global
  section $s \in \Aut\GG$ normalizes the constant sections $T$, its value
  $s(x) \in \TT_x \times^{T,\,\Ad} G$ normalizes $T=\TT_x \times^{T,\,\Ad} T$. Hence $s(x) \in
  \TT_x \times^{T,\,\Ad} N$ for the normalizer $N$ of $T$ in $G$, and therefore, $s$
  is a global section of
  \begin{equation*}
    \TT \times^{T,\,\Ad} N \subset \TT \times^{T,\,\Ad} G.
  \end{equation*}

  Let $p : N \to N/T=W_G$ be the canonical projection. Since $T$ acts
  trivially on $W_G$, the projection $p$ induces a map
  \begin{equation*}
    \TT \times^{T,\,\Ad} N \to X \times W_G
  \end{equation*}
  of group schemes over $X$. Here the image of $s$ is a global section of
  $X \times W_G$, which is automatically a constant $w \in W_G$ since $W_G$ is
  discrete. Therefore, $s$ is now a global section of the open subscheme
  \begin{equation*}
    \TT \times^{T,\,\Ad} p^{-1}(w) \subset \TT \times^{T,\,\Ad} N.
  \end{equation*}
  Hence $s : \TT \times^T G \to \TT \times^T G$ restricts to an isomorphism
  from the closed subscheme  $\TT \times^T T \subset \TT \times^T G$ to the
  closed subscheme $\TT \times^T p^{-1}(w) \subset \TT \times^T G$, where now
  $T$ acts on the $p^{-1}(w) \subset G$ once more by multiplication from the left. This
  restriction becomes an isomorphism of $T$-torsors when we let $T$ act on
  $p^{-1}(w)$ by multiplication from the right.

  Choose an element $n \in N(K)$ with $p(n) = w$, and define
  $\phi : T \to p^{-1}(w)$ by $\phi(t) = nt$. Then $\phi$ intertwines the two
  $T$-actions by multiplication from the right. Moreover, the relation
  $\phi(n^{-1}tnt')=t\phi(t')$ for $t,t' \in T$ shows that $\phi$ also
  intertwines the two $T$-actions where $t \in T$ acts on $T$ as
  multiplication by $\conj_{n^{-1}}(t)$ from the left, and on
  $p^{-1}(w)$ as multiplication by $t$ from the left. Hence $\phi$ induces an
  isomorphism of $T$-torsors
  \begin{equation*}
    (w^{-1})_*\TT = (\conj_{n^{-1}})_*\TT = \TT \times^{T,\,\conj_{n^{-1}}} T \to \TT \times^T p^{-1}(w).
  \end{equation*}
  As the latter is isomorphic to $\TT$ via $s$, we conclude that
  $(w^{-1})_* \TT$ is isomorphic to $\TT$.  Using assumption (ii), we deduce
  that $w=1$.

  Hence every automorphism of $\GG$ is a global section of
  $\TT \times^{T,\,\Ad} T = X \times T$, and thus an element of $T$.
\end{proof}

Now we assume that $S$ is a \emph{split} del Pezzo surface over $K$. Let $T$
be a split torus over $K$ with character lattice $\Lambda_T$, and let $\TT$ be
a $T$-torsor over $S$.  Recall from \cite{MR89f:11082} that the \emph{type} of
$\TT$ is the homomorphism $\tau: \Lambda_T \to \Lambda_S$ that sends a
character $\chi : T \to \Gm$ to the class $[\chi_*\TT]$ of the line bundle
given by the $\Gm$-torsor $\chi_*\TT$ over $S$. For a homomorphism
$\phi : T \to T'$ of split $K$-tori, it is easy to check that the extension of structure group
$\phi_*\TT$ has type
$\tau \circ \phi^*: \Lambda_{T'} \to \Lambda_T \to \Lambda_S$.

Let $G$ be a split reductive group over $K$ with maximal torus
$\iota: T \hookrightarrow G$ and resulting reduced root datum
$(\Phi_G \subset \Lambda_T,\Phi_G^\vee \subset \Lambda_T^\vee)$. We assume
that the type $\tau: \Lambda_T \to \Lambda_S$ of $\TT$ is an isomorphism from
this root datum to
$(\Phi_S \subset \Lambda_S,\Phi_S^\vee \subset \Lambda_S^\vee)$, i.e., $\tau$
is bijective, $\tau(\Phi_G)=\Phi_S$, and $\tau^*(\Phi_S^\vee)=\Phi_G^\vee$.
This determines $G$ up to isomorphism \cite[Expos\'e XXV,
Corollaire~1.2]{SGA3}. The assumption that $\tau$ is bijective also means
that $\TT$ is a universal torsor.

\begin{remark}\label{rem}
  This group $G$ can be described more explicitly as follows (see also
  \cite[\S 2]{MR1941576}).  The anticanonical class
  $-K_S \in \Lambda_S$ vanishes on all coroots and hence defines a
  character
  \begin{equation*}
    \chi := \tau^{-1}(-K_S) : G \to \Gm.
  \end{equation*}
  If the degree of $S$ is at most $6$, then the subgroup
  $\{\phi \in \Lambda_S^\vee \mid \phi(-K_S)=0\}$ is generated by the coroots
  $\alpha^\vee \in \Lambda_S^\vee$, so we obtain a short
  exact sequence
  \begin{equation*}
    1 \to [G,G] \to G \xrightarrow{\chi} \Gm \to 1,
  \end{equation*}
  in which $\ker(\chi)$ is simply connected. Therefore, the commutator
  subgroup $[G,G]$ is the semisimple and simply connected algebraic
  group of type $\mathbf{E}_8$, $\mathbf{E}_7$, $\mathbf{E}_6$,
  $\mathbf{D}_5$, $\mathbf{A}_4$, $\mathbf{A}_2+\mathbf{A}_1$ if the
  degree of $S$ is $1, 2, 3, 4, 5, 6$, respectively.

  In degree $7$, we have $G \cong \Gm \times \GL_2$. In degree
  $8$, we have $G \cong \GL_2$ if $S \cong \PP^1 \times \PP^1$,
  and $G \cong \Gm^2$ if $S$ is the blow-up of $\PP^2$ in a
  point. Finally, in degree $9$, we have $G \cong \Gm$.
\end{remark}

Let $\GG := \iota_* \TT$ denote the $G$-torsor over $S$ obtained from $\TT$ by
extension of structure group.  We can view the total space $\TT$ as contained
in $\GG$.

\begin{cor}\label{cor:IsomT_IsomG}
  Let $\TT_1$ be another $T$-torsor over $S$ of the same type $\tau$ as $\TT$, and 
  $\GG_1=\iota_*\TT_1$. Every isomorphism $\GG_1\to\GG$ of $G$-torsors over $S$
  restricts to an isomorphism $\TT_1\to\TT$ of $T$-torsors over $S$.
\end{cor}

\begin{proof}
  We can choose an isomorphism $\TT \cong \TT_1$ since both are torsors of the
  same type $\tau$ under the split torus $T$. Hence it suffices to verify that $\TT$
  satisfies the assumptions of Proposition~\ref{prop:IsomT_IsomG}.

  Indeed, for any root $\alpha \in \Phi_G$, the line bundle $L_\alpha$ over
  $S$ has as isomorphism class $\tau(\alpha)$ since $\TT$ has type $\tau$. In
  particular, the anticanonical degree of $L_{\alpha}$ is
  $(-K_{S}, \tau(\alpha)) = 0$, and $L_{\alpha} \not\cong \OO_{S}$.
  Hence $H^0( S, L_{\alpha}) = 0$ for each root $\alpha$.

  Furthermore, as $\TT$ has type $\tau$, the type of $w_*\TT$ is
  $\tau \circ w^* : \Lambda_T \to \Lambda_T \to \Lambda_S$. If $w \ne 1$, then
  $\tau \ne \tau \circ w^*$, and hence $\TT \not\cong w_*\TT$.
\end{proof}

\section{Proof of the main results}\label{sec:proof}

We use the notation of Section~\ref{sec:intro}. In particular, $S$ and $S_L$ are
both connected smooth projective schemes over $K$ (even though $S_L$ is not
geometrically connected over $K$ unless $L=K$), and the projection $\pi: S_L \to S$ is a
$K$-morphism. More precisely, $\pi$ is a Galois covering with group
\begin{equation*}
  \Sigma = \{\id_S \times \gamma^* \mid \gamma \in \Gamma\} \cong \Gamma^\op
\end{equation*}
in the sense of Definition~\ref{def:galois_mor}.

Let $\GG$ be a $G$-torsor over the $K$-scheme $S_L$, as in
Definition~\ref{def:G-torsor}. Then the same total space $\GG$ can also be
viewed as a $G_L$-torsor over the $L$-scheme $S_L$. Indeed, the required
action $\GG \times_L G_L \to \GG$ of $G_L$ comes from the given action
$\GG \times_K G \to \GG$ of $G$, since the two fiber products are canonically
isomorphic and the axioms are clearly equivalent. Conversely,
every $G_L$-torsor over the $L$-scheme $S_L$ can be viewed as a $G$-torsor
over the $K$-scheme $S_L$.

Let $T$ be a $K$-torus that splits over $L$ with character lattice $\Lambda_T$,
and let $\TT$ be a $T$-torsor over $S_L$. Then we can in particular view $\TT$
also as a $T_L$-torsor over $S_L$, which has a type $\tau : \Lambda_T \to
\Lambda_S$ in the sense of Section~\ref{sec:split}.

\begin{lemma}\label{lem:pullback_type}
  For $\gamma \in \Gamma$, put $\sigma = \id_S \times \gamma^* \in \Sigma$.
  Then the $T$-torsor $\sigma^*\TT$ over $S_L$ has type
  \begin{equation*}
    \sigma^* \circ \tau \circ (\gamma^T_*)^{-1}:
    \Lambda_T \to \Lambda_S,
  \end{equation*}
  where $\gamma^T_* : \Lambda_T \to \Lambda_T$ is induced by $\gamma$.
\end{lemma}

\begin{proof}
  Every element of $\Lambda_T$ has the form $\gamma^T_*\chi$ for a unique
  $\chi \in \Lambda_T$. The type of $\sigma^*\TT$ sends $\gamma^T_*\chi$ to
  the class of the line bundle
  \begin{equation*}
    (\gamma^T_*\chi)_*(\sigma^*\TT) \cong \sigma^*(\chi_*\TT)
  \end{equation*}
  over $S_L$, where the isomorphy is a special case of
  \eqref{eq:pullback_extension}. This class is $(\sigma^*\circ\tau)(\chi)$, as
  required.
\end{proof}

\begin{theorem}\label{thm:main_category}
  Let $\iota : T \hookrightarrow G$ be a maximal torus in a reductive group
  over $K$ such that $T$ splits over $L$, with resulting $\Gamma$-twisted root datum
  \begin{equation*}
    (\Phi_G \subset \Lambda_T,\Phi_G^\vee \subset \Lambda_T^\vee,
    \Gamma \times \Lambda_T \to \Lambda_T).
  \end{equation*}
  Let $\TT$ be a $T$-torsor over $S_L$ whose type
  $\tau: \Lambda_T \to \Lambda_S$ is an 
  isomorphism to the $\Gamma$-twisted root datum
  \begin{equation*}
    (\Phi_S \subset \Lambda_S,\Phi_S^\vee \subset \Lambda_S^\vee,
    \Gamma \times \Lambda_S \to \Lambda_S)
  \end{equation*}
  given by $S$.
  We consider the $G$-torsor $\GG := \iota_* \TT$ over $S_L$ obtained
  by extension of structure group.

  Then the functor $\iota_*$ from the groupoid of $T$-torsors $\TT^\circ$ over
  $S$ such that $\TT^\circ_L \cong \TT$ to the groupoid of $G$-torsors
  $\GG^\circ$ over $S$ such that $\GG^\circ_L \cong \GG$ is an equivalence of
  categories.
\end{theorem}

\begin{proof}
  Let $\TT^\circ$ be a $T$-torsor over $S$ such that $\TT^\circ_L \cong \TT$.
  Let $\GG^\circ:= \iota_* \TT^\circ$. Then the base change of $\GG^\circ$ to
  $L$ is
  \begin{equation*}
    \GG^\circ_L =(\iota_*\TT^\circ)_L\cong\iota_*(\TT^\circ_L)\cong\iota_*\TT = \GG.
  \end{equation*}
  Clearly any isomorphism $\phi$ between such $T$-torsors over $S$ induces an
  isomorphism $\iota_*\phi$ between $G$-torsors over $S$. Hence $\iota_*$ is
  indeed a functor.

  Let $\TT^\bullet$ be another $T$-torsor over $S$ such that
  $\TT^\bullet_L \cong \TT$, and $\GG^\bullet := \iota_*\TT^\bullet$. Let
  $\psi : \GG^\bullet \to \GG^\circ$ be an isomorphism of $G$-torsors over
  $S$.  Then Corollary~\ref{cor:IsomT_IsomG} applies to the base change
  \begin{equation*}
    \psi_L: \iota_*(\TT^\bullet_L) = \GG^\bullet_L \to \GG^\circ_L = \iota_*(\TT^\circ_L)
  \end{equation*}
  and shows that $\psi_L$ restricts to an isomorphism
  $\TT^\bullet_L \to \TT^\circ_L$.  In particular,
  $\psi_L(\TT^\bullet_L)=\TT^\circ_L$ as closed subschemes of $\GG^\circ_L$,
  and hence $\psi(\TT^\bullet)=\TT^\circ$ as closed subschemes of $\GG^\circ$.
  The restriction $\TT^\bullet \to \TT^\circ$ of $\psi$ becomes an isomorphism
  of $T$-torsors over $S_L$ after base change, and therefore is an isomorphism
  of $T$-torsors over $S$. This proves that the functor $\iota_*$ is fully
  faithful.

  Now assume that the $G$-torsor $\GG = \iota_*\TT$ descends to a $G$-torsor
  $\GG^\circ$ over $S$. As explained in Section~\ref{sec:torsors}, we have an
  isomorphism
  \begin{equation*}
    c_\sigma: \sigma^*\GG \to \GG
  \end{equation*}
  of $G$-torsors over $S_L$ for each $\sigma \in \Sigma$, satisfying the
  cocycle conditions \eqref{eq:cocycle}.  Here, we have
  $\sigma^*\GG = \iota_*\sigma^*\TT$. Since $\tau$ is $\Gamma$-equivariant,
  $\sigma^*\TT$ has the same type $\tau$ as $\TT$ due to
  Lemma~\ref{lem:pullback_type}. Therefore, Corollary~\ref{cor:IsomT_IsomG}
  shows that $c_\sigma$ restricts to an isomorphism
  \begin{equation*}
    c_\sigma': \sigma^*\TT \to \TT.
  \end{equation*}
  The $c_\sigma'$ satisfy the cocycle condition because the $c_\sigma$
  do. Therefore, $\TT$ descends to a $T$-torsor $\TT^\circ$ over $S$. By
  construction, $\iota_*\TT^\circ \cong \GG^\circ$. This proves that the
  functor $\iota_*$ is essentially surjective.
\end{proof}

\begin{cor}\label{cor:main_split}
  Let $\iota: T \hookrightarrow G$ be a split maximal torus in a split
  reductive group over $K$, with resulting root datum
  $(\Phi_G \subset \Lambda_T,\Phi_G^\vee \subset \Lambda_T^\vee)$.
  Let $\TT$ be a $T$-torsor over $S_L$ whose type $\tau: \Lambda_T \to \Lambda_S$
  is a $\ZZ$-linear isomorphism from this root datum to
  $(\Phi_S \subset \Lambda_S,\Phi_S^\vee \subset \Lambda_S^\vee)$.  We
  consider the $G$-torsor $\GG := \iota_* \TT$ over $S_L$ obtained by
  extension of structure group.

  Then the groupoid of $G$-torsors $\GG^\circ$ over $S$ such that
  $\GG^\circ_L \cong \GG$ is equivalent to the
  groupoid of universal torsors over $S$.
\end{cor}

\begin{proof}
  Let $N \subset G$ denote the normalizer of $T$ in $G$. The exact sequence
  \begin{equation*}
    1 \to T \to N \to W_G \to 1
  \end{equation*}
  defines the Weyl group $W_G$ of $G$. For $g \in G(L)$, we denote by
  \begin{equation*}
    \conj_g := g\cdot\_\cdot g^{-1}:  G_L \to G_L
  \end{equation*}
  the conjugation. For $n \in N(L)$, $\conj_n$ restricts to an automorphism of
  $T_L$, which depends only on the image of $n$ in $W_G$. This defines a left
  action of $W_G$ on $T_L$, and consequently a right action of $W_G$ on
  $\Lambda_T$, where $w \in W_G$ acts via the pullback
  $w^*: \Lambda_T \to \Lambda_T$ of characters. Sending $w \in W_G$ to
  $(w^{-1})^* : \Lambda_T \to \Lambda_T$, we identify $W_G$ with the Weyl
  group $W(\Phi_G) \subset \Aut \Lambda_T$ of our root datum of $G$.

  Since $\tau$ is an isomorphism of root data, it induces an isomorphism
  \begin{equation*}
    \tau_W : W_G \to W_S,
  \end{equation*}
  where $W_S:=W(\Phi_S) \subset \Aut \Lambda_S$ is the Weyl group of our root
  datum of $S$.

  The Galois action of $\Gamma$ on $\Lambda_S$ factors through the Weyl group action
  of $W_S$ on $\Lambda_S$ via a natural map
  \begin{equation}\label{eq:rho_S}
    \rho_S: \Gamma \to W_S.
  \end{equation}

  Applying \cite[Th\'eor\`eme~5.1(b)]{MR2139505} or
  \cite[Theorem~1.1]{MR2125504} to the semisimple commutator subgroup
  $G^c:=[G,G] \subset G$, there is a $g \in G^c(L) \subset G(L)$ such that,
  for each $\gamma \in \Gamma$, the element $g^{-1}\gamma(g) \in G(L)$ is in
  $N(L)$, and $\tau_W$ maps its image in $W_G$ to $\rho_S(\gamma)$. The former
  implies that there is a unique maximal torus $\iota': T' \hookrightarrow G$
  over $K$ such that $T'_L = gT_L g^{-1}$, and the latter says that the square
  on the right in the following diagram commutes:
  \begin{equation}\label{eq:Gal-equiv}
    \begin{tikzcd}
      \Lambda_{T'} \arrow[rr, "\conj_g^*"] \arrow[d, "\gamma^{T'}_*"]
      && \Lambda_T \arrow[rr, "\tau"] \arrow[d, "\conj_{\gamma(g^{-1})g}^*"] &&
      \Lambda_S \arrow[d, "\rho_S(\gamma)"]\\
      \Lambda_{T'} \arrow[rr, "\conj_g^*"] && \Lambda_T \arrow[rr, "\tau"]
      && \Lambda_S
    \end{tikzcd}
  \end{equation}
  In order to determine the Galois action $\gamma^{T'}_*$ coming from the fact
  that $T'$ is defined over $K$, we note that the commutative diagram
  \begin{equation*}
    \begin{tikzcd}
      G_L \arrow[rr, "\gamma(g)\cdot\_"] \arrow[d, "\gamma^*"]  && G_L
      \arrow[rr, "\_\cdot\gamma(g^{-1})"]  \arrow[d, "\gamma^*"] && G_L
      \arrow[d, "\gamma^*"]\\
      G_L \arrow[rr, "g\cdot\_"] && G_L\arrow[rr, "\_\cdot g^{-1}"] && G_L
    \end{tikzcd}
  \end{equation*}
  restricts to a commutative diagram
  \begin{equation*}
    \begin{tikzcd}
      T_L \arrow[rr, "\conj_{\gamma(g)}"] \arrow[d, "\gamma^*"]  && T_L' \arrow[d, "\gamma^*"] \\
      T_L \arrow[rr, "\conj_g"] && T_L',
    \end{tikzcd}
  \end{equation*}
  which in turn induces the commutative diagram of character lattices
  \begin{equation*}
    \begin{tikzcd}
      \Lambda_T    && \Lambda_{T'} \arrow[ll, swap, "\conj_{\gamma(g)}^*"] \\
      \Lambda_T \arrow[u, swap, "\gamma^T_*"] && \Lambda_{T'}. \arrow[ll, swap, "\conj_g^*"]\arrow[u, swap, "\gamma^{T'}_*"]
    \end{tikzcd}
  \end{equation*}
  Since $T$ is split over $K$, the map $\gamma^T_*$ here is the identity, and hence
  \begin{equation*}
    \gamma^{T'}_* = \conj_{g\gamma(g^{-1})}^*.
  \end{equation*}
  This shows that the square on the left in \eqref{eq:Gal-equiv} commutes as well.
  Therefore, the composition
  \begin{equation}\label{eq:Lambda_Gal}
    \Lambda_{T'} \xrightarrow{\conj_g^*} \Lambda_T \xrightarrow{\tau} \Lambda_S
  \end{equation}
  is Galois-equivariant. In particular, $T'$ is a N\'eron--Severi torus.
  
  Now Theorem~\ref{thm:main_category} applies to
  $\iota' : T' \hookrightarrow G$, the $T'$-torsor $\TT' := (\conj_g)_*\TT$
  and the resulting $G$-torsor $\GG' := \iota'_*\TT'$ over $S_L$.  Indeed,
  the type of $\TT'$ is by construction the composition \eqref{eq:Lambda_Gal}.
  This composition is Galois-equivariant, as we have just seen. It is also an
  isomorphism of root data since $\conj_g^*$ preserves the root datum of $G$
  and $\tau$ respects the relevant root data by assumption. So the hypotheses
  of Theorem~\ref{thm:main_category} are satisfied.

  Therefore, the functor $\iota'_*$ is an equivalence from the groupoid of
  $T'$-torsors $\TT^\circ$ over $S$ such that $\TT^\circ_L \cong \TT'$ to
  the groupoid of $G$-torsors $\GG^\circ$ over $S$ such that
  $\GG^\circ_L \cong \GG'$. The former groupoid is equivalent to the
  groupoid of universal torsors over $S$ because $T'$ is a N\'eron--Severi
  torus and the type of $\TT'$ is a Galois-equivariant isomorphism. The latter
  groupoid is the groupoid in the claim because $\GG' \cong \GG$. Indeed,
  $\GG' = (\conj_g)_*\GG$ because $\iota'_L\circ\conj_g =
  \conj_g\circ\iota_L$, and $\conj_g$ is an inner automorphism of $G_L$.
\end{proof}

\bibliographystyle{plain}

\bibliography{bundles}

\end{document}